\documentclass[11pt]{elsarticle}

\usepackage{amsmath,amsfonts,amssymb,latexsym,amsthm}
\usepackage{dsfont,mdwlist,bbold, bbm}

\def\qcut{ q_G^{Cut}} 
\def\qcutnorm{{ q}_G^{NCut}}
\def\qcutrel{q_G^{RCut}}

\def\uno{\mathbb 1}
\def\vol{\mathrm{vol}\,}
\def\t{\text{\itshape{\textsf{T}}}} 

\def\A{{\mathcal A}}
\def\L{{\mathcal L}}
\def\M{{\mathcal M}}

\newtheorem{theorem}{Theorem}[section]
\newtheorem{proposition}[theorem]{Proposition}
\newtheorem{lemma}[theorem]{Lemma}
\newtheorem{corollary}[theorem]{Corollary}

\theoremstyle{definition}
\newtheorem{definition}[theorem]{Definition}

\usepackage[dvipsnames]{xcolor}

\makeatletter
\def\ps@pprintTitle{%
	\let\@oddhead\@empty
	\let\@evenhead\@empty
	\def\@oddfoot{}%
	\let\@evenfoot\@oddfoot}
\makeatother

\begin{document} 

\title{Modularity bounds for clusters located by leading eigenvectors of the normalized  modularity~matrix}

\author{Dario Fasino\fnref{fndf}}

\address{Department of Chemistry, Physics, and Environment\\	University of Udine, Udine, Italy.}

\author{Francesco Tudisco\fnref{fnft}}

\address{Department of Mathematics and Computer Science,\\ Saarland University, Saarbr\"ucken, Germany}

\fntext[fndf]{The work of this author has been partially supported by INDAM-GNCS.}
\fntext[fnft]{The work of this author has been partially supported by the ERC Grant NOLEPRO.}

%
%

\begin{abstract}
Nodal theorems for generalized modularity matrices ensure that the cluster located by the positive entries of the leading eigenvector of various modularity matrices induces a connected subgraph.  In this paper we obtain lower bounds for the modularity of that set of nodes showing that, under certain conditions, the nodal domains induced by eigenvectors corresponding to highly positive eigenvalues of 
the normalized modularity matrix have indeed positive modularity, that is they can be recognized as modules inside the network. Moreover we establish  Cheeger-type inequalities for the cut-modularity of the graph, providing a theoretical support to the common understanding that highly positive eigenvalues of modularity matrices are related with the possibility of subdividing a network into communities.
\end{abstract}

\maketitle

\section{Introduction}

The study of community structures in complex networks is facing a significant growth, as observations on real life graphs reveal that many 
social, biological, and technological networks are intrinsically divided into clusters. 
Given a generic graph describing some kind of relationship
among actors of a complex network, community detection problems basically consist in discovering and revealing the groups (if any) in which the network is subdivided.


Modularity matrices, the main subject of investigation of the present work, are a relevant tool in the development of a sound theoretical background of community detection.   
Despite a number of modularity matrices has been proposed so far, see e.g., \cite{FT15} and the references therein, the original and most popular one was introduced by Newman and Girvan in \cite{newman-girvan} and is defined as a particular rank-one correction of the adjacency matrix. We shall refer to such matrix as the 
\textit{Newman--Girvan} (or \textit{unnormalized})
modularity matrix, and we will introduce consequently a normalized version of that matrix.

Spectral algorithms are widely applied to data clustering problems, including finding communities or partitions in graphs and networks. In the latter case, sign patterns in the
entries of certain eigenvectors of Laplacian matrices
are exploited to build vertex subsets, called
{\em nodal domains}, which often yield excellent solutions
to certain combinatorial problems related to 
the optimal partitioning of a given graph or network.

Analogously, nodal domains of modularity matrices play a crucial role in the community detection framework. A nodal domain theorem has been proved for these matrices \cite{FT14,FT15} showing the connectedness properties of nodal domains associated with their eigenvectors. The main results of this paper show that, under certain conditions, the nodal domains induced by eigenvectors corresponding to positive eigenvalues of the \textit{normalized} modularity matrix have indeed positive modularity, that is they can be recognized as modules inside the graph. Moreover we prove two Cheeger-type inequalities for the cut-modularity providing a theoretical support to the common understanding that highly positive eigenvalues of modularity matrices are related with the possibility of subdividing the graph into communities.

The paper is organized as follows. 
After fixing hereafter our notation and preliminary results, 
in Section \ref{sec:community-detection} we introduce with more detail the modularity based community detection problem, motivating our subsequent investigations. In Section \ref{sec:matrices} we discuss the unnormalized and normalized versions of the Newman--Girvan  modularity matrix, summarizing some of their main structural properties.  Subsequently, 
and 
we present our main results, concerning the relation between positive eigenvalues of the normalized modularity matrix and modules inside the graph. In particular in Section \ref{sec:main} we prove two Cheeger-type inequalities for the cut-modularity of the graph. Section \ref{sec:more} contains complementary results on 
modularity properties of nodal domains corresponding to
positive eigenvalues of the normalized modularity matrix.
We devote a brief final section to few relevant concluding remarks.

\subsection{Notations and preliminaries}
Hereafter, we give a brief review of standard concepts and symbols from algebraic graph theory that we will use throughout the paper.
We assume that $G = (V, E)$ 
is a simple connected graph, i.e., a
finite, undirected, unweighted graph 
without multiple edges, where
$V$ and $E$ are the vertex and edge sets. 
We always identify $V$ with $\{1,\dots,n\}$. 
We denote adjacency of vertices $x$ and $y$ interchangeably as 
$x \sim y$ or $xy\in E$.
Further definitions are listed
hereafter:

\begin{itemize}

\item For any $i\in V$, let $d_i$ denote its 
degree. 
Moreover, we let 
$d = (d_1,\ldots,d_n)^\t$, 
$\delta=(\sqrt{d_1},\ldots,\sqrt{d_n})^\t$,
$D = \mathrm{Diag}(d_1,\ldots,d_n)$.

\item The symbols $A$ and $\A$  denote 
the adjacency matrix of $G$ and its normalized counterpart, 
that is, $A\equiv(a_{ij})$
where $a_{ij} = 1$ if $ij \in E$, and $a_{ij}=0$ otherwise;
and $\A = D^{-1/2}AD^{-1/2}$. In particular, 
both $A$ and $\A$ are symmetric, irreducible, componentwise nonnegative matrices.

\item $\uno$ denotes the vector of all ones whose dimension depends on the context.

\item The cardinality of a set $S$ is denoted by $|S|$. 
In particular, 
$|V|=n$.

\item For any $S \subseteq \{1,\dots,n\}$ let $\uno_S$ be its \textit{characteristic vector}, defined as $(\uno_S)_i =1$ if $i \in S$ and $(\uno_S)_i=0$ otherwise.
Moreover, we denote by
$\bar S$ the complement $V \setminus S$, and let $\vol S = \sum_{i \in S}d_i$ be the {\em volume} of $S$. Correspondingly,  
$\vol V= \sum_{i \in V}d_i$ denotes the volume of the whole graph.






\item
For any subsets $S,T\subseteq V$ let
$$
   e(S,T) 
   = \uno_S^\t A \uno_T .
$$ 
For simplicity, we 
use the shorthands $e_{\mathrm{in}}(S) = e(S,S)$
and $e_{\mathrm{out}}(S) = e(S,\bar S)$,
so that $e_{\mathrm{in}}(S)$
is (twice) the number of inner-edges in $S$
and $e_{\mathrm{out}}(S)$ is the size of the edge-boundary of $S$.
We have also
$$
   \vol S = e_{\mathrm{in}}(S) + e_{\mathrm{out}}(S) .
$$

\item
A complete multipartite graph is a graph whose vertices
can be partitioned into pairwise disjoint subsets
$V_1,\ldots,V_k$ such that an edge
exists if and only if the two extremes belong to different
subsets, see e.g., \cite{Majstorovic2014}. In particular, if $k = n$ then $G$ is a complete graph,
while if $k=2$ and $V_1$ is a singleton then $G$ is a star.

\end{itemize}

\section{The community detection problem}
\label{sec:community-detection}

The discovery and description of communities in a graph is a central problem in modern graph analysis. Intuition suggests that a community (or cluster) in $G$ should be a possibly connected group of nodes whose internal edges outnumber those with the rest of the network. However there is no formal
definition of community. A survey of several proposed definitions of community can be found in \cite{santo-fortunato}, nonetheless as the author of that paper therein underlines, the global definition based on the modularity quality function is by far the most popular one. The modularity function was proposed by Newman and Girvan in \cite{newman-girvan} as a possible measure 
to quantify how much a given subset $S\subset V$ is a 
``good cluster''.
They postulate that $S$ is a cluster of nodes in $G$ if the induced subgraph $G(S)$ contains more edges than expected, if
edges were placed randomly.
Thus, they introduce the modularity function $Q(S)$ to  measure the difference between the actual and the expected number of edges in $G(S)$ so that a subset is a cluster if it has positive modularity. 
The precise definition 
is given by the following equivalent formulas:
\begin{equation}   \label{eq:def_Q(S)}
   Q(S) = e_{\mathrm{in}}(S) - \frac{(\vol S)^2}{\vol V}
   =\frac{\vol S \, \vol \bar S}{\vol V} - e_{\mathrm{out}}(S) .
\end{equation}
Note the equalities $Q(S) = Q(\bar S)$ and $Q(V) = 0$.
Undoubtedly, the modularity of a vertex set is
one of the most efficient indicators of its consistency as 
a community in $G$. For that reason, 
it is common practice to adopt the following definition:
\begin{definition}\label{def:mod-comm}
A subgraph of $G$ is a \textit{module} if its vertex set $S$ has positive modularity. If no ambiguity may occur, $S$ is called a module itself. 
\end{definition}

The usefulness of the previous definition
lies in the fact that, in practice,
if $G(S)$ is a connected module whose size is significant, 
then it can be recognized as a community. 

Definition \ref{def:mod-comm} leads naturally to an
efficient measure of a partitioning of $G$ into
modules. Indeed,
let $S_1,\ldots,S_k$ be a partition of $V$
into pairwise disjoint subsets. The {\em (normalized)
modularity of $S_1,\ldots,S_k$} is defined as
\begin{equation}\label{eq:q(partition)}
	q(S_1,\ldots,S_k) = \frac 1 {\vol V} 
	\sum_{i=1}^k Q(S_i) .
\end{equation}
The normalization factor $1/\vol V$ is somehow conventional. It has  been introduced in \cite{newman-eigenvectors,newman-girvan}
to settle the value of $q$ 
in a range independent on $G$ and $k$ and for compatibility with previous works.

The problem of partitioning a graph
into an arbitrary number of subrgaphs whose overall modularity
is maximized has received a considerable attention, 
not only in its applicative and computational aspects 
but also from the graph-theoretic point of view 
\cite{MR2942497,MR3089485}. 
The main contributions we propose in this work shall deal with the \textit{cut version} of the community detection problem,
that is 
the problem of finding a subset $S \subseteq V$ having maximal modularity (uniqueness   is not ensured in the general case).
To this end, it is worth to define the \textit{cut-modularity of the graph $G$} as the quantity
\begin{equation}\label{eq:modularity}
   \qcut = \max_{S\subseteq V} q(S,\bar S) 
    = \frac 2 {\vol V} \max_{S\subseteq V} Q(S) .
\end{equation}

It is well known that the optimization of the modularity 
function \eqref{eq:q(partition)}
presents some drawbacks when employed for finding a
partitioning of $G$ into modules, since small clusters tend to be subsumed by larger ones. Among the many techniques and 
variants of the Newman--Girvan modularity that have been devised to takle this issue,
which is widely known as {\em resolution limit},
here we borrow from \cite{Bolla2011} two weighted versions of the modularity function that play a relevant role in the subsequent discussion:
\begin{itemize}
	\item The {\em relative modularity} of $S\subseteq V$
	is $Q_{\mathrm{rel}}(S) = Q(S)/|S|$.
	This definition is naturally extended
	to the cut $\{S,\bar S\}$ as 
	\begin{equation}   \label{eq:q_rel}
	q_{\mathrm{rel}}(S,\bar S) = 
	Q_{\mathrm{rel}}(S) + Q_{\mathrm{rel}}(\bar S) = 
	Q(S) \frac{n}{|S||\bar S|} \, ,
	\end{equation}	
	which, in turn, leads to the definition of the \textit{relative cut-modularity of~$G$}
	$$\qcutrel = \max_{S\subseteq V}q_{\mathrm{rel}}(S,\bar S).$$
	\item The {\em normalized modularity} of $S\subseteq V$
	is defined as $Q_{\mathrm{norm}}(S) = Q(S)/\vol S$ and  that definition can be extended
	to the cut $\{S,\bar S\}$ as 
	\begin{equation}   \label{eq:q_norm}
	q_{\mathrm{norm}}(S,\bar S) = 
	Q_{\mathrm{norm}}(S) + Q_{\mathrm{norm}}(\bar S) = 
	Q(S) \frac{\vol V}{\vol S\vol\bar S} .
	\end{equation}
	As before we define the \textit{normalized cut-modularity of the graph $G$ as}
	$$\qcutnorm = \max_{S\subseteq V}q_{\mathrm{norm}}(S,\bar S).$$
\end{itemize}
Straightforward computations ensure
	\begin{equation*}
	\frac {2\, \qcutrel}{n\,  d_{\max}}  \leq \qcut \leq \frac \qcutrel 2,  \qquad 	\frac {2 \, \qcutnorm}{\vol V} \leq \qcut \leq \frac{\qcutnorm}{2} .
	\end{equation*}

\section{Modularity matrices and their properties}\label{sec:matrices}

The probably best known methods for detecting a subset whose modularity well approximates 
the cut-modularity of $G$ are based on the idea of spectral partitioning and are related with an important rank-one correction of the adjacency matrix,  known as 
\textit{the Newman--Girvan modularity matrix}. 
In analogy with the graph Laplacians, in this section we define two different modularity matrices, describing a number of relevant structural properties. 

\subsection{The Newman--Girvan modularity matrix}

Given a graph $G$ and the associated adjacency matrix $A$, let $d=A\uno$ be the degree vector of $G$, and $\vol V = \sum_i d_i$ be its volume. 
The \textit{unnormalized modularity matrix} of $G$ has been introduced in \cite{newman-eigenvectors} as 
the following rank one perturbation of $A$:
\begin{equation}   \label{eq:M}
   M = A - \frac{1}{\vol V} dd^\t .
\end{equation}
For any $S\subseteq V$ let $\uno_S$ be 
its characteristic vector: $(\uno_S)_i = 1$ if $i\in S$
and $(\uno_S)_i = 0$ otherwise.
With the help of these notations we can express $Q(S)$ as 
\begin{equation}   \label{eq:Q(S)}
   Q(S) = \uno_S^\t M \uno_S 
\end{equation}
  The following proposition summarizes some basics  properties of $M$:
 \begin{proposition}   \label{prop:M}
 	The matrix $M$ satisfies the following properties:
 	\begin{enumerate}
 		\item $M$ is symmetric and $\uno \in \ker(M)$.
 		\item If $m_1\geq \dots \geq m_n$ are the eigenvalues of $M$ and $\alpha_1\geq \dots \geq \alpha_n$ those of $A$, then $\alpha_1 \geq m_1 \geq \alpha_2 \geq m_2 \geq \dots \geq \alpha_n \geq m_n$.
 		\item $0$ is a simple eigenvalue of $M$ if and only if $A$ is nonsingular.
\item 
The rightmost eigenvalue of $M$ is 
nonnegative, and is zero if 
and only if $G$ is a complete multipartite graph.
 	\end{enumerate}
 \end{proposition}
 \begin{proof}
Point 1 is
revealed by a direct computation. Point 2 is a direct consequence of the variational characterization of the eigenvelaues of symmetric matrices, see e.g., \cite{wilkinson}. To show point \emph 3 we observe that the multipliticy of the zero eigenvalue of $M$ is one plus the dimension of the kernel of $A$. Indeed consider the diagonal matrix $\mathit \Delta = \mathrm{Diag}(1/\sqrt{d_1}, \dots, 1/\sqrt{d_n})$ and let $\delta = \mathit \Delta d$. Then $\mathit \Delta M \mathit \Delta \delta = 0$ and   $\mathit \Delta A \mathit \Delta \delta = \delta$. Therefore the multiplicity of the zero eigenvalue of $\mathit \Delta M \mathit \Delta $ is the multiplicity of the zero eigenvalue of $\mathit \Delta A \mathit \Delta$ plus one. This proves point \emph 3 as the multiplicity of $0$ is invariant under matrix congruences.
Point \emph 4 is a rephrasing of 
Theorem 1.1 in \cite{Majstorovic2014}. See also 
\cite[Thm.\ 11]{MR3338340}.
\end{proof}

The modularity matrix $M$ is at the  basis of many  spectral methods for community detection, and the eigenstructure of $M$ can be used to describe  clustering properties of graphs. In particular,  the nodal domains associated to its principal eigenvectors  cover a special role, as they are often good candidates for leading modules inside $G$.  A number of results relating algebraic properties of $M$ to communities in $G$ have appeared in recent literature \cite{Bolla2011,MR3338340,FT14, FT15,Majstorovic2014}, the  forthcoming Theorem \ref{thm:advanced-M} summarizes those among them which to our opinion are most relevant.

 As it often plays a special role in the algebraic analysis of the modular structure of $G$, the rightmost nonzero eigenvalue of $M$ deserves a the special symbol, borrowed from \cite{FT14} and therein named \textit{algebraic modularity}:
\begin{equation}   \label{eq:m_G}
   m_G = \max_{\substack{v \in \mathbbm R^n \\ v^\t \uno =0}}
   \frac{v^\t M v}{v^\t v}\, .
\end{equation} 
Already at this stage intuition suggests that a close relation should exists between $m_G$ and the cut-modularity 
\eqref{eq:modularity}, and that the subsets $S\subseteq V$  having positive modularity should be related with  positive eigenvalues of $M$. 
The following theorem  summarizes some important eigenproperties of $M$ that have been proven in recent literature,
see in particular, 
\cite{MR3338340, FT14, Majstorovic2014}.

\begin{theorem} \label{thm:advanced-M} The  matrix $M$ satisfies the following properties:
\begin{enumerate}

\item $m_G<\rho(A)$ and, 
if $d$ is not an eigenvector of $A$, then $m_G$ is simple.

\item 
If $G$ is not a complete graph
or a complete multipartite graph then $m_G = \lambda_1(M)$,
the rightmost eigenvalue of $M$, and is positive. 
If $G$ is a star then $m_G = \lambda_2(M)$, the second rightmost eigenvalue of $M$, and is negative.
Otherwise (that is, if $G$ is a complete graph
or a complete multipartite graph which is not a star)
$m_G = 0$.

\item Let $\left<d\right>=\vol V/n$ be the average degree of $G$, then $m_G \geq 2 \left<d\right> \qcut$.

\item Let $\{S_1, \dots, S_k\}$ be a partition that maximizes the quantity in \eqref{eq:q(partition)}, which has minimal cardinality, and which is made up entirely by modules. Then $k-1$ does not exceed the number of positive eigenvalues of $M$.

\item Let $u$ be an eigenvector associated with $m_G$ such that $d^\t u \geq 0$. If $m_G$ is simple and it is not an eigenvalue of $A$ then the subgraph induced by the subset $S_+=\{i \mid u_i \geq  0\}$ is connected.
		
\end{enumerate}
\end{theorem}

For any $S\subseteq V$ let $v_S = \uno_S - \frac{|S|}{n}\uno$.
The following identities are readily obtained:
$$
   v_S^\t \uno = 0 , \qquad 
   v_S^\t v_S =  \frac{|S||\bar S|}{n} , \qquad
   v_S^\t M v_S = Q(S) , \qquad
   q_{\mathrm{rel}}(S,\bar S) = \frac{v_S^\t M v_S}{v_S^\t v_S}.
$$
Hence, the combinatorial problem of finding the cut $\{S,\bar S\}$
with largest relative modularity has a natural continuous relaxation in the maximization of the Rayleigh quotient
$v^\t Mv/v^\t v$ over the subspace orthogonal to $\uno$,
that is, the algebraic modularity defined in \eqref{eq:m_G}.
We have the immediate consequence
$$
   \qcutrel \leq m_G .
$$


\subsection{The normalized modularity matrix}

In analogy with the renowed normalized Laplacian matrix of a graph, we let $\A = D^{-1/2}AD^{-1/2}$ be the normalized adjacency matrix and define the \textit{normalized modularity matrix} of $G$ as 
$$
   \M = D^{-1/2}M D^{-1/2} =
   \A - \frac{1}{\vol V}\delta \delta^\t
$$
where $\delta = (\sqrt{d_1},\ldots,\sqrt{d_n})^\t$
and $M$ is as in \eqref{eq:M}.
The matrix $\M$
appeared recently in the community detection literature, and in  
various other network related questions as the analysis of quasi-randomness properties of graphs with given degree sequences,
see \cite{Bolla2011,MR2371048,FT15} and \cite[Chap.\ 5]{chung}.
Several basics properties of $\M$ can be immediately observed; we collect some of them hereafter. 

\begin{proposition} \label{oss:M}
	The matrix $\mathcal M$ satisfies the following properties:
\begin{enumerate}
\item  $\M$ has a zero eigenvalue with corresponding eigenvector $\delta$. 
 
\item The matrices $\M$ and $\A$ 
coincide over the space orthogonal to $\delta$.
That is, $\M v = \A v$ for all $v\in\langle\delta\rangle^\perp$.

\item The eigenvalues of $\M$ belong to the interval $[-1, 1]$.
Moreover, $0$ is a simple eigenvalue of $\M$ if and only if 
$\A$ is nonsingular.

\item 
If $G$ is connected then $1$ is not an eigevalue of $\M$.
Furthermore, if $G$ is not bipartite then $-1$ is not an eigevalue of $\M$.

\end{enumerate}
\end{proposition}

\begin{proof}
Straightforward computations  show that $\A \delta = \delta$
and $\M\delta =0$. Since $\A \geq O$ and $\delta \geq 0$, 
Perron--Frobenius theory leads us to deduce that $\rho(\A)=1$
is an eigenvalue of $\A$. Therefore,
if $\A = \sum_{i=1}^n \lambda_i q_iq_i^\t$ is a spectral 
decomposition of $\A$ with the eigenvalues in nonincreasing order,
$\lambda_1 \geq \ldots \geq \lambda_n$, then we can assume
$\lambda_1 = 1$, $|\lambda_i| \leq 1$ for $i > 1$, 
and $q_1$ parallel to $\delta$.
In particular,
$\delta \delta^\t/\vol V$ is the orthogonal projector on the eigenspace 
spanned by $q_1$, since $\delta^\t \delta =\vol V$. Consequently, 
$\M = \sum_{i=2}^n \lambda_i q_iq_i^\t$ is a spectral 
decomposition of $\M$ and we easily deduce points \emph 2 and \emph 3.
Incidentally, this proves that $\M$ and $\A$ are simultaneously diagonalizable.
If $G$ is connected then $\A$ is irreducible and
$\lambda_1$ is simple, that is $1 > \lambda_2$.
Furthermore, if $G$ is not bipartite then
$\A$ is also primitive and $|\lambda_i| < 1$ for $i > 1$, 
and the proof is complete.
\end{proof}

\medskip
The normalized modularity 
\eqref{eq:q_norm} of a cut $\{S,\bar S\}$
can be naturally defined in terms of $\M$. 
In fact, given any $S \subseteq V$, consider the vector 
\begin{equation}   \label{eq:v_S}
   v_S = D^{1/2}(\uno_S - c \uno),
   \qquad c = \vol S/\vol V .
\end{equation}
Simple computations prove that 
$$
   \delta^Tv_S = 0 , \qquad 
   v_S^\t v_S = \frac{\vol S \, \vol \bar S}{\vol V} . 
$$
Moreover,
$$
   \frac{v_S^\t \M v_S}{v_S^\t v_S} 
   =  \frac{(\uno_S - c\uno)^\t M 
   (\uno_S - c\uno)}{v_S^\t v_S}  
   =  \frac{\uno_S^\t M \uno_S}
   {\vol S \,\vol \bar S}  \vol V 
   = q_{\mathrm{norm}}(S,\bar S) .
$$
It follows that the problem of 
computing the normalized  cut-modularity of $G$
 can be stated in terms of $\M$. Indeed, if $\mathcal V_n$ is the set of $n$-vectors having the form \eqref{eq:v_S} for some 
$S\subset V$,
then $v_S$ is a generic vector in $\mathcal V_n$, implying that
\begin{equation}\label{eq:q_G}
   \qcutnorm =  \max_{v \in \mathcal V_n}
   \frac{v^\t \M v}{v^\t v} 
\end{equation}
and of course, if $\hat v$ is the vector realizing the maximum in \eqref{eq:q_G}, then the set $\hat S = \{i \mid \hat v_i > 0\}$ defines the optimal cut. 
As for the unnormalized case, it is worth defining the  \textit{normalized algebraic modularity}:
\begin{equation}\label{eq:mu(G)}
   \mu_G = 
   \max_{\substack{v \in \mathbbm R^n \\ v^\t \delta=0}}
   \frac{v^\t \M v}{v^\t v} .
\end{equation}
Note that  \eqref{eq:mu(G)} is a relaxed version of \eqref{eq:q_G}. In particular,
\begin{equation}   \label{eq:ineq_q_norm}
   \qcutnorm \leq \mu_G .
\end{equation}
Since $\M$ is real symmetric we immediately note that $\mu_G$ coincides with the largest eigenvalue of $\M$ after deflation of the invariant subspace spanned by $\delta$. 
Therefore, if $-1\leq \mu_n \leq \cdots \leq \mu_1 \leq 1$ are the eigenvalues of $\M$, then $\mu_1 = \max\{0, \mu_G\}$.
Furthermore, since $M$ and $\M$ are related by a congruence
transform, point \emph 2 of Theorem \ref{thm:advanced-M} 
leads us to the following result: 
\begin{corollary}   \label{cor:mu}
If $G$ is not a star then $\mu_G = \mu_1$, the rightmost eigenvalue of $\M$. Moreover, 
$\mu_G > 0$ if and only if 
$G$ is not a complete graph or a complete multipartite graph.
\end{corollary}

\section{Cheeger-type inequalities}\label{sec:main} 

As we already discussed above, both heuristics and intuition suggest that $\mu_G$ quantifies the cut-modularity of the graph, and can be used to approximate $\qcutnorm$. While the upper bound $\qcutnorm\leq \mu_G$
has been shown in \eqref{eq:ineq_q_norm} by simple arguments, a converse relation, bounding $\qcutnorm$ from below in terms of $\mu_G$, is not that easy. 
In fact, there it is possible that $\mu_G > 0$ while $\qcutnorm < 0$,
as shown experimentally in \cite{MR3338340}.
Theorems \ref{teo:simple-cheeger} and \ref{teo:cheeger} 
contribute to this question stating  lower (and upper) bounds of $\qcutnorm$ in terms of spectral properties of of $\M$. 

The  {conductance} (or {sparsity}, or Cheeger constant) 
$h_G$ is one of  the best known topological invariants of a graph $G$. 
For $S\subset V$ let
$$
   h(S)= \frac{e_{\mathrm{out}}(S)}{\min\{\vol S, \vol \bar S\}}\, ,
$$
the so-called conductance of $S$. Then,
the conductance of $G$ is defined as $h_G = \min_{S\subset V}h(S)$. 
Such quantity plays a fundamental role in graph partitioning problems \cite[Chap.\ 11]{newman-book}, 
in isoperimetric problems \cite[Chap.\ 2]{chung}, 
mixing properties of random walks, combinatorics,
and in various other areas of mathematics and computer science. 
A renowned result in  graph theory, known as \textit{Cheeger inequality},
relates the conductance of $G$ and the smallest positive eigenvalue of 
the normalized Laplacian matrix $\L = I - \A$. 

If $0= \lambda_1 < \lambda_2 \leq \dots \leq \lambda_n \leq 2$ are the eigenvalues of  $\L$,
the Cheeger inequality states that  
$$
   \textstyle{\frac {1}2\lambda_2\leq 
   h_G \leq \sqrt{2\lambda_2}} .
$$ 
Actually, Chung \cite{chung} improved the upper bound to
$h_G \leq \sqrt{\lambda_2(2-\lambda_2)}$.
Let $v$ be an eigenvector of $\L$ corresponding to $\lambda_2$ and
consider the equality 
$\L = I - \A = I - \M + \delta\delta^\t/\delta^\t \delta$.
Since $\L \delta = 0$, we have $\delta^\t v = 0$.
By Courant's minimax principle and \eqref{eq:mu(G)},
$$
   \lambda_2 = \min_{v:\delta^\t v = 0}\frac{v^\t\L v}{v^\t v}
   = 1 - \max_{v:\delta^\t v = 0}\frac{v^\t\M v}{v^\t v}
   = 1 - \mu_G .
$$
In particular, from Corollary \ref{cor:mu}
we obtain that, if $G$ is not a star then $1 - \lambda_2$
is the rightmost eigenvalue of $M$. 
A direct application of the Cheeger inequality yields
the following estimates for $\qcutnorm$.

\begin{theorem} \label{teo:simple-cheeger} 
Let $\mu_1$ be the rightmost eigenvalue of $\M$.
If $G$ is not a star then 
$$
   1-2\sqrt{1-\mu_1^2}\leq \qcutnorm \leq \mu_1 .
$$
\end{theorem}

\begin{proof}
Recalling \eqref{eq:def_Q(S)} and \eqref{eq:q_norm}, we have
\begin{align*}
   q_{\mathrm{norm}}(S,\bar S) & = \frac{\vol V}{\vol S\vol \bar S}Q(S) \\
   &= 1 -  \frac{\vol V}{\vol S\vol \bar S}e_{\mathrm{out}}(S)\geq 1- 2 h(S) ,
\end{align*}
since $\vol V/\vol S \vol \bar S \leq 2/ \min\{\vol S, \vol \bar S\}$.   
By maximizing over $S$ we eventually get
\begin{equation*}
	\qcutnorm = 
	\max_{S \subset V}q_{\mathrm{norm}}(S,\bar S) 
	\geq 1-2h_G \geq 1-2\sqrt{(1-\mu_G)(1+\mu_G)} . 
\end{equation*}
By hypothesis, $\mu_G = \mu_1$.
The upper bound comes from \eqref{eq:ineq_q_norm}.
\end{proof}

Extensive research on Cheeger-type results by many authors
suggests that
no substantial improvements on the lower bound in Theorem \ref{teo:simple-cheeger}
can be obtained without additional information on $G$,
although explicit examples of graph sequences 
proving optimality of that bound are not known.
However, the forthcoming result shows that, almost surely,
$1 - \mu_1$ can be a much better estimate to $1 - \qcutnorm$
than expected, in  particular, when the entries of an eigenvector of $\mu_1$
cluster around two values. 
We will make use of the following lemma, whose simple proof is omitted for brevity:

\begin{lemma}   \label{lem:norma1}
If $\sum_{i=1}^n \alpha_i = 0$ then 
$\sum_{i: \alpha_i>0} \alpha_i = \frac12 \sum_{i=1}^n |\alpha_i|$.
\end{lemma}

\begin{theorem}  \label{teo:cheeger}
Let $\mu_1$ be the rightmost eigenvalue of $\M$.
Suppose that $\mu_1$ has an eigenvector $x$ without zero entries.
Then there exists a constant $C>0$, not depending on $\mu_1$, such that 
$$
   1 - C(1 - \mu_1) \leq \qcutnorm .
$$
\end{theorem}

\begin{proof}
Let $v$ be an eigenvector of $\M$ corresponding to $\mu_1$ 
and let $z = D^{-1/2}v$. 
Note that $v$ is orthogonal to the vector 
$\delta = (\sqrt{d_1},\ldots,\sqrt{d_n})^\t$, since the latter
is an eigenvector of $\M$ associated to $0$. 
Consequently, $z$
is orthogonal to the degree vector: 
$d^\t z = \delta^\t D^{1/2}z = \delta^\t v = 0$.
Hence, 
$$
   \mu = \frac{v^\t\M v}{v^\t v} = 
   \frac{v^\t\A v}{v^\t v} = \frac{z^\t A v}{z^\t D z} 
   = 1 - \frac{z^\t L z}{z^\t D z} ,
$$
where $L = D - A$ is the Laplacian matrix of $G$.
We have 
$$
   z^\t L z = \sum_{ij\in E}(z_i-z_j)^2 ,
$$
where the sum runs over the edges of the graph, each edge
being counted only once. On the other hand,
$$
   z^\t D z = \sum_{i=1}^n d_iz_i^2 .
$$   
For notational simplicity, we use the shorthands
$s = \vol S$, $\bar s = \vol \bar S$, 
and $\nu = s + \bar s = \vol V$.
Consider the nodal domain $S = \{i:v_i \geq 0\}$ and let $x$ be the step vector $x = p\uno_S + q\uno_{\bar S}$
which minimizes the weighted distance
$$
   \|D^{1/2}(x - z)\|_2^2 = \sum_{i=1}^n d_i(x_i - z_i)^2
   = \sum_{i\in S} d_i(p - z_i)^2 
   + \sum_{i\in \bar S} d_i(q - z_i)^2 .
$$
Simple computations show that the minimum is attained when
$$
   p = \Big(\sum_{i\in S} d_iz_i\Big) /s , \qquad
   q = \Big(\sum_{i\in \bar S} d_iz_i\Big) /\bar s .
$$
Observe that $p$ and $q$ are weighted averages of the values
$z_i$ for $i\in S$ and $i\in \bar S$, respectively.
With the notation $c = \sum_{i\in S} d_iz_i$,
from the orthogonality condition $d^\t z = 0$ 
and Lemma \ref{lem:norma1}
we deduce the simpler formulas $p = c/s$ and $q = -c/\bar s$. 
For later reference, we remark the identities
\begin{equation}   \label{eq:pq}
   p-q = \frac{c\nu}{s\bar s} , \qquad
   p^2s + q^2\bar s = \nu\frac{(c\nu)^2}{(s\bar s)^2} .
\end{equation}
Incidentally, we note that, apart of a constant,
the vector $D^{1/2}x$ coincides with the vector in \eqref{eq:v_S}.
Moreover, it is not hard to recognize that, if $G$ is disconnected
then the vector $D^{1/2}x$ is an eigenvector of $\M$ 
associated to the eigenvalue $1$.
Our subsequent arguments are based on the intuition that,
if $z$ is a small perturbation of $x$ then 
$S$ is weakly linked to $\bar S$.
Let $r\geq 1$ be a number such that 
$$
   r^{-1} \leq z_i/x_i \leq r , \qquad i = 1,\ldots,n .
$$
In fact, if $z_i > 0$ then $x_i = p > 0$,
whereas $z_i < 0$ implies $x_i = q < 0$.
Hence, if $ij\in E$ is an edge joining a node in $S$
with a node in $\bar S$ we have
$|z_i - z_j| \geq (p-q)/r$. Consequently,
$$
   z^\t L z = \sum_{ij\in E}(z_i-z_j)^2 
   \geq r^{-2}(p-q)^2 e_{\mathrm{out}}(S) ,
$$
by neglecting all contributions from edges lying entirely
inside $S$ or $\bar S$. Moreover,
$$
   z^\t D z = \sum_{i=1}^n d_iz_i^2 \leq
   r^2\Big( \sum_{i\in S}p^2 d_i + \sum_{i\in \bar S}p^2 d_i\Big)
   = r^2 (p^2s + q^2 \bar s) .
$$ 
Consider the equality
$e_{\mathrm{out}}(S) = 
(1-q_{\mathrm{norm}}(S,\bar S))s\bar s/\nu$.
Using \eqref{eq:pq} and simplifying we get 
$$
   1 - \mu = \frac{z^\t L z}{z^\t Dz} \geq 
   \frac{1}{r^4\nu} e_{\mathrm{out}}(S) 
   = \frac{s\bar s}{r^4\nu^2} (1-q_{\mathrm{norm}}(S,\bar S))   
   \geq \frac{1}{4r^4} (1-\qcutnorm) ,
$$
owing to $s\bar s/\nu^2 \geq \frac14$. 
\end{proof}

\section{Modules from nodal domains}   \label{sec:more}

Theorems \ref{teo:simple-cheeger} and \ref{teo:cheeger} 
state in particular that if $\mu_G$ is 
sufficiently close to $1$, then 
the cut-modularity of $G$ is positive and thus there exists   
a bipartition of $V$ into $\{S,\bar S\}$ such that both $G(S)$ and $G(\bar S)$ are modules. 
Of course such bipartition is not unique in the general case. 
The forthcoming theorems strengthen this claim by showing that, 
if a positive eigenvalue $\mu$ of $\M$ is large enough, then 
we can explicitly exhibit a cut 
$\{S,\bar S\}$ with positive modularity, by defining it in terms of a nodal domain induced by an eigenvector corresponding to $\mu$. 
 
 Given a nonzero vector $v\in\mathbbm{R}^n$ the 
 subgraph $G(S)$ induced by the 
 set $S = \{i : v_i \geq 0\}$ is a {\em nodal domain} of $v$
 \cite{nodal-domain-theorem,duval-nodal-domains}.
 This fundamental definition admits obvious variations
 (for example, inequality can be strict, or reversed)
 and, since the seminal papers by Fiedler \cite{fiedler-connectivity,fiedler-vector},
 it has become a major tool for  
 spectral methods in community detection and graph partitioning
 \cite{newman-eigenvectors,powers-graph-eigenvector,Schaeffer2007}.
 Indeed, nodal domains of eigenvectors
 of modularity matrices are commonly utilized 
 in order to localize modules inside a network. If $v$ is an eigenvector corresponding to $\mu_G$, it has been shown in  \cite{FT15} that $S = \{i: v_i \geq 0\}$ induces a connected subgraph $G(S)$. The following Theorems \ref{teo:q(S_+)} and \ref{teo:q(S_+)altro} provide additional information on $G(S)$ as they show that, if $\mu_G$ is large enough, then the subgraph $G(S)$ is a module.

\begin{theorem}   \label{teo:q(S_+)}
Let $v$ be a normalized eigenvector of $\M$ corresponding to 
a positive eigenvalue $\mu$, that is,
$\M v = \mu v$ with $\|v\|_2 = 1$.
Let $S = \{i \mid v_i \geq 0\}$. If 
$$
   \mu > \frac{(\vol S)^2 + (\vol \bar S)^2}{\vol V}  
   \, \max_{i \in V}\frac{v_i^2}{d_i}
$$
then  $Q(S)>0$.
\end{theorem}

\begin{proof}
Recalling Proposition \ref{oss:M}, we have
that $v$ 
is orthogonal to  $\delta$, which implies in turn $\M v = \A v$ 
and $\mu = v^\t \M v = v^\t \A v$. 
Define the set $\mathcal{I}_+ = (S\times S)\cup(\bar S\times\bar S)$.
Note that $v_iv_j \geq 0$ whenever $(i,j)\in\mathcal{I}_+$.
Using entrywise nonnegativity of $\A$ we obtain
$$ 
   \mu =
   v^\t \A v \leq \sum_{(i,j) \in \mathcal{I}_+}
   v_i v_j \A_{ij} 
   \leq \left(
   \max_{i \in V}\frac{|v_i|}{\delta_i}\right)^2 
   \sum_{(i,j) \in \mathcal{I}_+}\delta_i \delta_j \A_{ij} .
$$ 
Since $\delta_i \delta_j \A_{ij} = A_{ij}$, the 
rightmost summations yield
$$
   \sum_{(i,j) \in \mathcal{I}_+}A_{ij}
   = \uno_{S}^\t A \uno_{S}+
   \uno_{\bar S}^\t A \uno_{\bar S} 
   = e_{\mathrm{in}}(S) + e_{\mathrm{in}}(\bar S)  .
$$
Let us set $C^2 = (\max_{i \in V}|v_i|/\delta_i)^2$.
Owing to the equalities
$Q(S) = e_{\mathrm{in}}(S) - (\vol S)^2/\vol V$
and $Q(S) = Q(\bar S)$ we have 
$$
   \mu  \leq C^2 
   \big(e_{\mathrm{in}}(S) + e_{\mathrm{in}}(\bar S) \big) 
   = C^2
   \Big( 2Q(S) + \frac{(\vol S)^2 + (\vol \bar S)^2}{\vol V} \Big).
$$  
By rearranging terms,
$$
   2C^2Q(S) \geq \mu - C^2 \frac{(\vol S)^2 + (\vol \bar S)^2}{\vol V} ,
$$
and the claim follows.
\end{proof}

With respect to the quantity $\max_i v_i^2/d_i$ appearing
in the preceding theorem, consider that if $G$ is
$k$-regular (that is, $d_i = k$ for every $i\in V$)
then $v_i = n^{-\frac12}$ and $\vol V = kn$. After simple
passages the aforementioned lower bound
for $\mu$ becomes $(|S|^2 + |\bar S|^2)/n^2$,
a number which is strictly smaller than $1$.

%

\begin{theorem}   \label{teo:q(S_+)altro}
Let $v$ be any real eigenvector of $\M$ corresponding to 
a positive eigenvalue $\mu $, that is,
$\M v = \mu  v$.
Let $S = \{i \mid v_i \geq 0\}$ and let
$\cos\theta$ be the cosine of the acute angle between
the vectors $|v| = (|v_1|,\ldots,|v_n|)^\t$ and 
$\delta = (\sqrt{d_1},\ldots,\sqrt{d_n})^\t$. If 
$$
   \mu  + 1 > 4\frac{\vol S \, \vol \bar S}{(\vol V)^2}
   \, \frac{1}{\cos^2\theta}
$$
then  $Q(S)>0$.
\end{theorem}

\begin{proof}
Let $s = D^{1/2}\uno_S$, that is
$$
   s_i = \begin{cases}
   \delta_i & v_i \geq 0 , \\
   0 & \hbox{otherwise.} \end{cases}
$$
Observe that 
$\|s\|_2^2 = \sum_{i\in S} d_i = \vol S$
and $\delta^\t s = \vol S$ too.
Since $v$ is orthogonal to 
$\delta = (\sqrt{d_1},\ldots,\sqrt{d_n})^\t$, 
there exist scalars $\alpha$, $\beta$, $\gamma$ 
such that we 
have the orthogonal decomposition
\begin{equation}   \label{eq:dec_s}
   s = \alpha \frac{1}{\|\delta\|_2} \delta
   + \beta \frac{1}{\|v\|_2} v + \gamma w 
\end{equation}
for some normalized vector $w\in\mathbbm{R}^n$ 
orthogonal to both $\delta$ and $v$.
The coefficients in \eqref{eq:dec_s}
own the following explicit formulas:
$$ 
   \alpha  = \frac{1}{\|\delta\|_2} \delta^\t s
   = \frac{\vol S}{\sqrt{\vol V}} , \qquad
   \beta  =  \frac{v^\t s}{\|v\|_2} ,
$$ 
and moreover,
\begin{align*}
   \gamma^2 & = \|s\|_2^2 - \alpha^2 - \beta^2 
   = \vol S - \frac{(\vol S)^2}{\vol V} - 
   \beta^2 \\
   & = \frac{\vol S\, \vol \bar S}{\vol V} - 
   \beta^2.
\end{align*}
Owing to the fact that the spectrum of $\M$
is included in $[-1,1]$ and the assumption $\|w\|_2 = 1$
we have $w^\t\M w \geq -1$.
Hence, from \eqref{eq:dec_s} we obtain
\begin{align*}
   Q(S) & = \uno_S^\t M \uno_S = s^\t \M s \\
   & \geq \alpha^2 \cdot 0 + \beta^2 \, \mu  - \gamma^2 
   = \beta^2 (\mu  + 1) - 
   \frac{\vol S\, \vol \bar S}{\vol V} .
\end{align*}
Thus, if
$$
   \mu  + 1 > \frac{\vol S\, \vol \bar S}{\beta^2 \, \vol V}
$$
then $Q(S) > 0$.
Moreover, using the orthogonality
$\delta^\t v = 0$ 
and Lemma \ref{lem:norma1}
we obtain
$$
   \cos\theta = 
   \frac{\sum_{i\in V} \delta_i |v_i|}{\|v\|_2\|\delta\|_2} =
   \frac{2 \sum_{i\in S} \delta_i v_i}{\|v\|_2\sqrt{\vol V}} =
   2 \frac{v^\t s}{\|v\|_2\sqrt{\vol V}} ,
$$
whence $\beta = \frac12 (\cos\theta)\sqrt{\vol V}$
and the proof is complete. 
\end{proof}

From the straightforward bound 
$$\textstyle{\vol S \, \vol\bar S /(\vol V)^2 \leq \frac14}$$
and the equality $\cos^{-2}\theta - 1 = \tan^2\theta$, 
we derive the following condition.

\begin{corollary}
In the same notations of Theorem \ref{teo:q(S_+)altro},
if $\mu  > \tan^2	\theta$ then $Q(S) > 0$.
\end{corollary}

\section{Concluding remarks} 

Community detection is a major task in modern complex network analysis and the matrix approach to such problem is quite popular and powerful. In this work we formulate the modularity of a cut in terms of a quadratic form associated with the normalized modularity matrix, and we 
provide theoretical supports to the common understanding that highly positive eigenvalues of the normalized modularity matrix imply the presence of communities in $G$.  In particular we show that, 
if that matrix has an eigenvalue close to $1$ then the nodal domains corresponding to that eigenvalue have positive modularity and, moreover,
can produce good estimates of the optimal cut-modularity. 

 As recent advances in spectral graph theory have shown higher order Cheeger inequalities in terms of higher order eigenvalues of the graph Laplacian \cite{HeinTudisco2015a,trvisan-higher-cut}, we believe that deeper spectral based investigations could reveal more precise relations between the magnitude and the number of positive eigenvalues of the modularity matrices and the presence of communities in the network.

\bibliographystyle{plain}
\bibliography{./networks_ref_aug15}

\end{document}